\documentclass[12pt]{amsart}

\usepackage{graphicx,amsmath,amscd,amsfonts,amsthm,amssymb,verbatim,stmaryrd,fullpage}
\newtheorem{theorem}{Theorem}[section]
\newtheorem{lemma}[theorem]{Lemma}
\newtheorem{prop}[theorem]{Proposition}

\newtheorem{cor}[theorem]{Corollary}

\theoremstyle{remark}

\newcommand{\R}{\mathbb R}

\newcommand{\Z}{\mathbb Z}
\newcommand{\Q}{\mathbb Q}

\newcommand{\A}{\mathbb A}

\newcommand{\bD}{{\bf D}}
\newcommand{\bZ}{{\bf Z}}
\newcommand{\bG}{{\bf G}}

\newcommand{\g}{\mathfrak g}

\newcommand{\gk}{\mathfrak k}
\newcommand{\ga}{\mathfrak a}

\newcommand{\gn}{\mathfrak n}

\newcommand{\cO}{\mathcal O}

\newcommand{\cL}{\mathcal L}

\newcommand{\cT}{\mathcal T}
\newcommand{\cM}{\mathcal M}

\newcommand{\cH}{\mathcal H}

\newcommand{\ad}{\text{ad}}
\newcommand{\Ad}{\text{Ad}}

\newcommand{\ord}{\text{ord}}

\newcommand{\be}{\begin{equation}}
\newcommand{\ee}{\end{equation}}
\newcommand{\bes}{\begin{equation*}}
\newcommand{\ees}{\end{equation*}}

\newcommand{\ba}{\begin{eqnarray}}
\newcommand{\ea}{\end{eqnarray}}
\newcommand{\bas}{\begin{eqnarray*}}
\newcommand{\eas}{\end{eqnarray*}}

\title{Restrictions of $SL_3$ Maass forms to maximal flat subspaces}

\author{Simon Marshall}
\address{Department of Mathematics\\
University of Wisconsin -- Madison\\
480 Lincoln Drive\\
Madison\\
WI 53706, USA}
\email{marshall@math.wisc.edu}
\thanks{Supported by NSF grant DMS-1201321.}

\begin{document}

\begin{abstract}
Let $\psi$ be a Hecke-Maass cusp form on a cubic central simple algebra over $\Q$.  We apply arithmetic amplification to improve the local bound for the $L^2$ norm of $\psi$ restricted to maximal flat subspaces.
\end{abstract}

\maketitle

\section{Introduction}
\label{sec1}

Let $S$ be the globally symmetric space $SL_3(\R) / SO(3)$.  Let $\Gamma \subset SL_3(\R)$ be an arithmetic congruence lattice arising from a central simple algebra over $\Q$, and let $X = \Gamma \backslash S$ (see Section \ref{sec2} for definitions).  $X$ is a Riemannian orbifold of finite volume, and a manifold if $\Gamma$ is torsion free.  Let $\psi$ be a Hecke-Maass cusp form on $X$, that is to say a cuspidal eigenfunction of the full ring of invariant differential operators and the Hecke operators.  We assume that $\| \psi \|_2 = 1$.  We also assume that the spectral parameter of $\psi$ is of the form $t \lambda$, where $t > 1$ and $\lambda \in B^*$, and $B^*$ is a fixed compact regular subset of $\ga^*$.

Let $\Omega \subseteq X$ be a compact set.  Let $E \subset \Omega$ be a ball of radius 1 inside a maximal flat subspace of $S$.  It was proven in Theorem 1.2 of \cite{Ma2} that $\| \psi|_E \|_2 \ll_{B^*, \Omega} t^{3/4}$, and moreover that this bound is sharp on the compact globally symmetric space $SU(3) / SO(3)$ which is dual to $S$.  Note that Theorem 3 of \cite{BGT} and the $L^\infty$ bound of \cite{Sa} also provide bounds of this form with exponents of 1 and $3/2$ respectively.  In this paper, we apply arithmetic amplification to improve this exponent further.

\begin{theorem}
\label{main}

For any $0 < \delta < 1/12000$, there is $C = C(B^*, \Omega, \delta) > 0$ such that $\| \psi|_E \|_2 \le C t^{3/4 - \delta}$.

\end{theorem}

We have chosen this particular restriction problem because it is one of only two cases in which we can restrict to a maximal flat subspace and observe a regime change, or kink point, in the local bound for the $L^p$ norm of the restricted eigenfunction.  The other case is a geodesic on a surface, treated in \cite{Ma1}; see Theorem 1.2 of \cite{Ma2} for the proof of this classification.  Both of these features simplify the problem, as the flatness prevents us from having to use the nonabelian Fourier transform on the subspace, and the presence of a regieme change lies behind the strong bound we are able to prove for the `off-diagonal' oscillatory integrals in Proposition \ref{Ibound}.  It would be interesting to see whether Theorem \ref{main} could be used to prove a power saving for the global $L^p$ norms of $\psi$ for small $p$ as in \cite{BS, Bo, So}.

{\bf Acknowledgements:}  We would like to thank Valentin Blomer, Zeev Rudnick, and Matthew Stover for helpful comments, and the referee for a careful reading of the manuscript.

\subsection{Outline of the proof}
\label{outline}

We shall prove Theorem \ref{main} in the same way as the main theorem of \cite{Ma1}.  We shall use standard notations defined in Section \ref{sec2}.  Choose a compact set $\Omega \subset SL_3(\R)$.  If we fix a real-valued function $b \in C^\infty_0(\ga)$, it suffices to estimate the norm of $b\psi = b(H) \psi(g \exp(H)) \in L^2(\ga)$, uniformly for $g \in \Omega$.  We shall do this by estimating $\langle b\psi, \phi \rangle$ for $\phi \in L^2(\ga)$ of norm one.

To prove such an estimate, we shall construct an $SO(3)$-biinvariant function $k_t \in C^\infty_0(SL_3(R))$ whose Harish-Chandra transform $h_t$ is non-negative on the spectral parameters of unitary representations, satisfies $h_t(t\lambda) \ge 1$, and is concentrated near $Wt\lambda$.  Form the kernel function

\bes
K(x,y) = \sum_{\gamma \in \Gamma} \overline{k_t}(x^{-1} \gamma y)
\ees
on $\Gamma \backslash SL_3(\R)$.  The integral operator on $L^2(X)$ with kernel $K(x,y)$ is then an approximate spectral projector onto eigenfunctions with spectral parameter $t\lambda$.  We write the spectral decomposition of $L^2(X)$ as

\bes
L^2(X) = \int V_\pi d\pi,
\ees
where each $V_\pi$ is a one dimensional space spanned by an eigenform $\psi_\pi$ (which may be an Eisenstein series or an iterated residue thereof).  We assume that $\psi \in \{ \psi_\pi\}$.  If we let the spectral paramater of $\psi_\pi$ be $\lambda_\pi$, the kernel $K(x,y)$ has the spectral expansion

\bes
K(x,y) = \int h_t(\lambda_\pi)\psi_\pi(x) \overline{\psi_\pi(y)} d\pi.
\ees
If we form the integral

\bes
\iint_\ga \overline{b\phi(H_1)} b\phi(H_2) K( g \exp(H_1), g \exp(H_2) ) dH_1 dH_2,
\ees
substituting the spectral and geometric expansions of $K(x,y)$ gives

\begin{multline}
\label{specint1}
\sum_{\gamma \in \Gamma} \iint_\ga \overline{b\phi(H_1)} b\phi(H_2) \overline{k_t}( \exp(-H_1) g^{-1} \gamma g \exp(H_2) ) dH_1 dH_2 \\
= \int h_t(\lambda_\pi) | \langle \psi_\pi, b \phi \rangle |^2 d\pi \ge | \langle \psi, b \phi \rangle |^2,
\end{multline}
which is analogous to a relative trace formula for the subgroup $A$.

We shall need to apply the identity (\ref{specint1}) in two different ways.  We define the Fourier transform on $L^2(\ga)$ by

\bes
\widehat{f}(\mu) = \int_\ga f(H) e^{-i \mu(H)} dH.
\ees
Let $\beta$ be a parameter satisfying $1 \le \beta \le t^{1/2}$.  For $\mu \in \ga^*$, define $H(\mu, \beta) \subset L^2(\ga)$ to be the space of functions whose Fourier support lies in the ball of radius $\beta$ around $\mu$ with respect to the norm $\| \cdot \|$ defined in Section \ref{sec22}.  Define

\bes
H_\beta = \bigoplus_{w \in W} H(wt\lambda, \beta).
\ees
If we choose $k_t$ to be sufficiently small that only the identity contributes to (\ref{specint1}), and let $\phi \in H_\beta^\perp$ have norm 1, analysing the resulting integral allows us to prove the bound

\be
\label{offspec1}
\langle b\psi, \phi \rangle \ll_\epsilon t^{3/4} \beta^{-1/4 + \epsilon}.
\ee
This shows that the main contribution to $\| b \psi \|_2$ comes from those frequencies near $W t\lambda$.  The exact statement we prove is Proposition \ref{L2offspec}.  The bound (\ref{offspec1}) is purely local, and does not use the arithmeticity of $\Gamma$ or $\psi$.  However, it does require asymptotics for the spherical functions on $SL_3$ that are strongly uniform in the group variable, which are taken from Theorems 1.3 and 1.4 of \cite{Ma2}.\\

We bound $\langle b\psi, \phi \rangle$ for $\phi \in H(wt\lambda, \beta)$ of norm 1 by applying a technique known as arithmetic amplification, which amounts to introducing a Hecke operator in the identity (\ref{specint1}).  This method was introduced as a way of bounding $L^\infty$ norms of Maass forms on $GL_2$ by Iwaniec and Sarnak in \cite{IS}.  It has since been used by many authors to bound $L^\infty$ norms of Maass forms on other groups (see for instance \cite{BM1,BM2,BP,HRR}), which may be thought of as bounding the $L^2$ norms of their restrictions to points.  This paper is the second time the method has been applied to restrictions to submanifolds of positive dimension, with the first being \cite{Ma1}.

We shall give an outline of the amplification method.  We let $\cT$ be a Hecke operator, and apply $\cT \cT^*$ to $K(x,y)$ in the first variable.  The identity corresponding to (\ref{specint1}) is now

\begin{multline}
\label{specint2}
\sum_{\gamma \in \cT \cT^* \Gamma} C(\gamma) \iint_\ga \overline{b\phi(H_1)} b\phi(H_2) \overline{k_t}( \exp(-H_1) g^{-1} \gamma g \exp(H_2) ) dH_1 dH_2 \\
= \int h_t(\lambda_\pi) | \langle \cT \psi_\pi, b \phi \rangle |^2 d\pi \ge | \langle \cT \psi, b \phi \rangle |^2,
\end{multline}
where $\cT \cT^* \Gamma$ is the set of isometries appearing in $\cT \cT^*$, and $C(\gamma)$ is the coefficient of $\gamma \in \cT \cT^* \Gamma$.  We choose $\cT$ so that its eigenvalue on $\psi$ is large, and it should have small eigenvalues on the remaining $\psi_\pi$ by the orthogonality of systems of Hecke eigenvalues.  The term `amplification' comes from this way in which $\cT$ picks out $\psi$ from the collection $\{ \psi_\pi \}$.  We prove in Section \ref{sec6} that the oscillatory integrals appearing in (\ref{specint2}) are small unless $g^{-1} \gamma g$ is very close to $A$.  In Section \ref{sec4}, we use a diophantine argument to show that there are few $\gamma$ such that this happens.  This argument is taken from an unpublished paper of Lior Silberman and Akshay Venkatesh, and we thank the authors for permission to reproduce it here.  Combining these gives the bound

\be
\label{onspec1}
\langle b\psi, \phi \rangle \ll t^{3/4 - \delta} \beta^{9/4},
\ee
where we may take any $0 < \delta < 1/1200$.  The exact statement we prove is Proposition \ref{L2onspec}.  Combining (\ref{offspec1}) and (\ref{onspec1}) with $\beta = t^{2\delta/5}$ gives Theorem \ref{main}.

\section{Notation}
\label{sec2}

Throughout the paper, the notation $A \ll B$ will mean that there is a positive constant $C$ such that $|A| \le CB$, and $A \sim B$ will mean that there are positive constants $C_1$ and $C_2$ such that $C_1 B \le A \le C_2 B$.

\subsection{Division algebras and adelic groups}

We let $\A$ be the adeles of $\Q$, and $\A_f$ the finite adeles.  Let $D$ be a cubic central simple algebra over $\Q$.  We denote the reduced norm on $D$ by $\text{nr}$, and denote the kernel of $\text{nr}$ by $D^1$.  We let $\bD^\times$ and $\bD^1$ be the algebraic groups over $\Q$ such that $\bD^\times(\Q) = D^\times$ and $\bD^1(\Q) = D^1$.  We denote the center of $\bD^\times$ by $\bZ$, and define $\bG = \bD / \bZ$.  We denote $D \otimes_\Q \Q_v$ by $D_v$.  We denote $\bD^\times(\Q_v)$ by $D_v^\times$, and likewise for the other groups we have introduced.  Let $S_f$ be a finite set of finite places containing all places at which $D$ is ramified.  We choose an isomorphism $\phi_v : D_v \simeq M_3(\Q_v)$ for every $v \notin S_f$.  We shall implicitly identify $D_v^\times$ with $GL_3(\Q_v)$ via $\phi_v$ for $v \notin S_f$.

Let $R \subset D$ be a maximal order.  We define $R_p = R \otimes_\Z \Z_p \subset D_p$ for every prime $p$.  $R_p$ is a maximal order in $D_p$ for all $p$, and for $p \notin S_f$ we choose $\phi_p$ so that $\phi_p(R_p) = M_3(\Z_p)$.  We let $K_f = \otimes_p K_p$ be a maximal compact subgroup of $\bD^\times(\A_f)$. We assume that $K_p \subseteq R_p^\times$ for all $p$, and that $K_p = R_p^\times$ for $p \notin S_f$.  We define $K_\infty = \phi_\infty^{-1}( SO(3))$, and let $K = K_\infty \otimes K_f$.

For each $p$, we define $dg_p$ to be the Haar measure on $D_p^\times$ that assigns mass 1 to $R_p^\times$.  We define $dz_p$ to be the Haar measure on $Z_p$ that assigns mass 1 to $\Z_p^\times$, and let $d\overline{g}_p$ be the quotient measure on $G_p$.  We choose Haar measures $dg_\infty$, $dz_\infty$, and $d\overline{g}_\infty$ on $D^\times_\infty$, $Z_\infty$, and $G_\infty$, and denote the product measures by $dg$, $dz$, and $d\overline{g}$.

Define $X = \bD^\times(\Q) \backslash \bD^\times(\A) / K \bZ(\A)$.  $X$ is compact iff $D \not\simeq M_3(\Q)$.  If $X$ is an orbifold, we resolve any difficulties in talking about differential operators on $X$ by passing to a cover that is a manifold.

\subsection{Lie groups and algebras}
\label{sec22}

We define $A$ to be the subgroup of $SL_3(\R)$ consisting of diagonal matrices with positive entries, and let $Z_A$ be the centraliser of $A$ in $M_3(\R)$.  We define $N$ to be the subgroup of strictly upper triangular matrices.  We denote the Lie algebras of $SL_3(\R)$, $N$, $A$, and $SO(3)$ by $\g$, $\gn$, $\ga$ and $\gk$ respectively.  We denote the roots of $\ga$ in $\g$ by $\Delta$, and the set of positive roots corresponding to $\gn$ by $\Delta^+$.  We denote the set of regular and singular points in $\ga$ and $\ga^*$ by $\ga_r$, $\ga_s$, etc.  We let $M$ and $M'$ be the centraliser and normaliser of $\ga$ in $SO(3)$, and define the Weyl group $W = M' / M$.

We equip $M_3(\R)$ with the standard Euclidean norm as a 9-dimensional vector space, which we denote by $\| \cdot \|$.  We obtain a positive definite norm on $\g$ from $\| \cdot \|$ under the natural restriction, as well as a norm on $\ga^*$ by duality.  We shall also denote these norms by $\| \cdot \|$, the particular one we are using will be clear from the context.  We let $B_\gn$, $B_\ga$, and $B_\gk$ be the unit balls in $\gn$, $\ga$, and $\gk$ with respect to $\| \cdot \|$.  We let $d( \cdot, \cdot)$ be the left invariant metric on $SL_3(\R)$ associated to $\| \cdot \|$.  The Killing form on $\g$ will always be denoted by $\langle \, , \rangle$, and we give $S$ the Riemannian structure determined by $\langle \, , \rangle$.

The identification $\phi_\infty$ gives an identification of $G_\infty$ with $SL_3(\R)$, and we implicitly transfer all the definitions we have made here to $G_\infty$.

\subsection{Hecke algebras}

If $p \notin S_f$, we let $\cH_p$ denote the space of functions in $C^\infty_0(GL_3(\Q_p))$ that are bi-invariant under $K_p$.  It is an algebra under convolution with respect to the measure $dg_p$.  If $\varphi \in \cH_p$, we denote its adjoint operator by $\varphi^*$, which satisfies $\varphi^*(g) = \overline{\varphi(g^{-1})}$.  If $(a,b,c) \in \Z^3$, we define $K_p(a,b,c) \subset GL_3(\Q_p)$ to be the double coset

\bes
K_p(a,b,c) = K_p \left( \begin{array}{ccc} p^a & & \\ & p^b & \\ & & p^c \end{array} \right) K_p,
\ees
and let $\Phi_p(a,b,c) \in \cH_p$ be the characteristic function of $K_p(a,b,c)$.

We let $\cH$ be the space of functions in $C^\infty_0( \bD^\times(\A_f))$ that are bi-invariant under $K_f$.  We shall frequently identify $\cH_p$ with a subagebra of $\cH$ in the natural way.  If $(a, b, c) \in (\Q^\times)^3$, define $K(a, b, c) \subset \bD^\times(\A_f)$ to be

\bes
K(a, b, c) = \bigotimes_{p \in S_f} K_p \otimes \bigotimes_{p \notin S_f} K_p( \ord_p(a), \ord_p(b), \ord_p(c) ),
\ees
and let $\Phi(a,b,c) \in \cH$ be the characteristic function of $K(a, b, c)$.  We will sometimes implicitly identify $\Phi(a,b,c)$ and $K(a,b,c)$ with their images in $\bG$ under central integration and projection.  The action of $\varphi \in \cH$ on a function $f$ on $\bD^\times(\Q) \backslash \bD^\times(\A)$ is given by the usual formula

\bes
\varphi f(x) = f * \varphi^{\vee}(x) = \int_{\bD^\times(\A_f)} f(xg) \varphi(g) dg,
\ees
where $\varphi^\vee(g) = \varphi(g^{-1})$.

\subsection{Spherical functions}

If $\mu \in \ga^*$, we define $\varphi_\mu$ to be the corresponding spherical function on $G_\infty$.  If $k \in C^\infty_0(G_\infty)$, we define its Harish-Chandra transform by

\bes
\widehat{k}(\mu) = \int_{G_\infty} k(g) \varphi_{-\mu}(g) d\overline{g}_\infty.
\ees
If $k \in C^\infty_0(G_\infty)$ is $K$-biinvariant, we denote its adjoint by $k^*$, which satisfies $k^*(g) = \overline{k(g^{-1})}$.

\subsection{Maass forms}

Let $\psi$ be a cuspidal Hecke-Maass form on $X$, that is to say an eigenfunction of the ring of invariant differential operators and of the Hecke algebras $\cH_p$ for $p \notin S_f$.  We assume that $\| \psi \|_2 = 1$.  We let the spectral parameter of $\psi$ be $t\lambda$, where $t > 0$ and $\lambda \in B^*$, and $B^*$ is a fixed compact subset of $\ga^*_r$.

\section{Constructing an amplifier}
\label{sec3}

Let $p \notin S_f$ be a prime.  In this section, we construct an element $T_p \in \cH_p$ that will form part of the amplifier $\cT$.  We begin with the following relation in $\cH_p$.

\begin{lemma}
\label{Heckerel}

We have

\bes
\Phi_p(1,0,0) * \Phi_p(1,1,0) = \Phi_p(2,1,0) + (p^2 + p + 1) \Phi_p(1,1,1).
\ees

\end{lemma}

\begin{proof}

As $\Phi_p(1,0,0) * \Phi_p(1,1,0)$ must be supported on those double cosets $K_p(a,b,c)$ with $a+b+c = 3$ and $a$, $b$, $c \ge 0$, we must have

\be
\label{Heckeexpand1}
\Phi_p(1,0,0) * \Phi_p(1,1,0) = a \Phi(1,1,1) + b \Phi_p(2,1,0) + c \Phi_p(3,0,0)
\ee
for some $a$, $b$, and $c \in \R$.  Taking adjoints gives

\bes
\Phi_p(-1,0,0) * \Phi_p(-1,-1,0) = a \Phi_p(-1,-1,-1) + b \Phi_p(-2,-1,0) + c \Phi_p(-3,0,0),
\ees
and it follows that $c = 0$ as there can be no matrices whose entries have denominator $p^3$ in the support of $\Phi_p(-1,0,0) * \Phi_p(-1,-1,0)$.

If we translate (\ref{Heckeexpand1}) by $p^{-1} I$ we obtain

\be
\label{Heckeexpand2}
\Phi_p(1,0,0) * \Phi_p(-1,0,0) = a \Phi(0,0,0) + b \Phi_p(1,0,-1),
\ee
and it may be easily shown that

\bes
dg_p(K_p(1,0,0)) = p^2 + p + 1, \quad dg_p(K_p(1, 0, -1)) = (p^2+p)(p^2+p+1).
\ees
Evaluating (\ref{Heckeexpand2}) at the identity gives

\bes
a = \Phi_p(1,0,0) * \Phi_p(-1,0,0)(e) = \| \Phi_p(1,0,0) \|_2^2 = p^2 + p + 1.
\ees
Integrating (\ref{Heckeexpand2}) over $GL_3(\Q_p)$ gives

\bes
dg_p(K_p(1,0,0))^2 = a + b \, dg_p(K_p(1, 0, -1)),
\ees
and it follows that $b = 1$.

\end{proof}

Lemma \ref{Heckerel} implies that if we define $a(\psi,p)$ and $b(\psi,p)$ by

\bes
\Phi_p(1,0,0) \psi = a(\psi,p) p \psi, \quad \Phi_p(2,1,0) \psi = b(\psi,p) p^2 \psi,
\ees
then we cannot have both $|a(\psi,p)| \le 1/2$ and $|b(\psi,p)| \le 1/2$.  We define

\be
\label{Tpdef}
T_p = \Big\{ \begin{array}{ll} \Phi_p(1,0,0)/a(\psi,p)p & \text{if} \quad |a(\psi,p)| > 1/2, \\
\Phi_p(2,1,0) / b(\psi,p) p^2 & \text{otherwise}. \end{array}
\ee
It follows that $T_p \psi = \psi$ for all $p \notin S_f$.  We shall need the following bound for the coefficients in the expansion of $T_p T_p^*$.

\begin{lemma}
\label{Tpexpand}

Write

\bes
T_p T_p^* = \sum_{a \ge b \ge c} \alpha(a,b,c) \Phi_p(a,b,c).
\ees
If $\alpha(a,b,c) \neq 0$ then we have

\be
\label{Heckecoeffbd}
\alpha(a,b,c) \ll p^{c-a}
\ee
where the implied constant is absolute.  Moreover, one of the three pairs of inequalities

\begin{align}
\label{abcbd}
-1 \le a, b, c \le 2& \quad and \quad  a + b + c = 2, \\
\notag
-2 \le a, b, c \le 2& \quad and \quad a + b + c = 0, \\
\notag
-2 \le a, b, c \le 1& \quad and \quad a + b + c = -2,
\end{align}
holds.

\end{lemma}

\begin{proof}

If we define

\bes
\Phi_p = p^{-1} \Phi_p(1,0,0) + p^{-2} \Phi_p(2,1,0)
\ees
and write

\be
\label{Heckeexpand3}
\Phi_p * \Phi_p^* = \sum_{a \ge b \ge c} \beta(a,b,c) \Phi_p(a, b, c),
\ee
then we have $\beta(a,b,c) \ge 0$ and $|\alpha(a,b,c)| \le 2\beta(a,b,c)$.  The bounds (\ref{abcbd}) on $a$, $b$, and $c$ may be proven on a case by case basis by taking determinants, and taking adjoints and considering denominators as in Lemma \ref{Heckerel}.

We shall bound $\beta(a,b,c)$ using the spherical transform.  Define $A_p : GL_3(\Q_p) \rightarrow \Z^3$ to be the $p$-adic Iwasawa $A$ co-ordinate with respeect to the standard collection of subgroups, and let

\begin{align*}
\rho : \Z^3 & \rightarrow \Z \\
(a,b,c) & \mapsto a-c
\end{align*}
be the half sum of the positive roots.  We define the function $\varphi_0$ by

\bes
\varphi_0(g) = \int_{K_p} p^{\rho(A_p(kg))} dk_p,
\ees
so that $\varphi_0$ is the spherical function with trivial Satake parameter.  It follows from Proposition 7.2 of \cite{Co} that

\bes
\Phi_p(1,0,0) * \varphi_0 = \Phi_p(1,1,0) * \varphi_0 = 3p \varphi_0.
\ees
Combining this with Lemma \ref{Heckerel} gives $\Phi_p(2,1,0) * \varphi_0 = (8p^2 - p - 1) \varphi_0$, and so

\bes
(\Phi_p * \Phi_p^*) * \varphi_0 = (11 - p^{-1} - p^{-2})^2 \varphi_0.
\ees
Taking only one term in the expansion (\ref{Heckeexpand3}) and evaluating at the identity gives

\bes
\beta(a,b,c) (\Phi_p(a, b, c) * \varphi_0)(e) \le (11 - p^{-1} - p^{-2})^2.
\ees
We have

\begin{align*}
(\Phi_p(a, b, c) * \varphi_0)(e) & = \int_{K_p(a,b,c)} \varphi_0(g^{-1}) dg_p \\
& = \int_{K_p(a,b,c)} p^{\rho(A(g^{-1}))} dg_p \\
& \ge p^{a-c},
\end{align*}
and (\ref{Heckecoeffbd}) now follows.

\end{proof}

\section{Amplification of periods along flats}
\label{sec4}

Fix a real-valued function $b \in C^\infty_0(\ga)$ with $\text{supp}(b) \subseteq B_\ga$.  Let $\Omega \subset G_\infty$ be a compact set, let $g_0 \in \Omega$, and let $b \psi$ denote the function $b(H) \psi( g_0 \exp(H)) \in L^2(\ga)$.  This means that we are restricting ourselves to flat subspaces contained in the identity component of $X$.  This entails no loss of generality, as we may treat the other connected components by first translating $\psi$ by an element of $\bG(\A_f)$.  (In fact, the Hasse-Schilling norm theorem (see Theorem 33.15 of \cite{Re}) and the fact that $\bD^1$ satisfies strong approximation imply that an element of $(R \otimes_\Z \A_f)^\times$ will suffice.)  If we let $w \in W$ and $\phi \in H(wt\lambda, \beta)$ satisfy $\| \phi \|_2 = 1$, where $H(wt\lambda, \beta)$ is as in Section \ref{outline}, we shall prove the following bound for $\langle b \psi, \phi \rangle$.

\begin{prop}
\label{L2onspec}

For any $\epsilon > 0$ there is $C = C(B^*, \Omega, \epsilon) > 0$ such that

\bes
\langle b \psi, \phi \rangle \le C t^{3/4 - 1/1200 + \epsilon} \beta^{9/4}.
\ees

\end{prop}

To prove Proposition \ref{L2onspec}, we may assume without loss of generality that $w = 1$ as the other cases are identical.  Fix a real-valued non-negative function $h \in C^\infty(\ga^*)$ of Paley-Wiener type that satisfies $h(0) = 1$.  Define $h_t^0$ by 

\bes
h_t^0(\mu) = \sum_{w \in W} h( w\mu - t\lambda),
\ees
and let $k^0_t$ be the $K$-biinvariant function on $G_\infty$ with Harish-Chandra transform $h_t^0$.  The Paley-Wiener theorem of Gangolli \cite{Ga} implies that $k^0_t$ is of compact support that is uniform in $t \lambda$, and may be chosen to be arbitrarily small.  We define $k_t = k_t^0 * (k_t^0)^*$, and $h_t = \widehat{k_t}$.  Assume that $\textrm{supp}(k_t( \exp(H))) \subseteq B_\ga$.

Let $1 \le N \le t$ be an integer to be chosen later, and define $\cT$ to be the Hecke operator

\bes
\cT = \sum_{1 \le p \le N} T_p
\ees
where $T_p$ is as in (\ref{Tpdef}).  If we define $\cT(\psi)$ to be the scalar by which $\cT$ acts on $\psi$, the equations $T_p \psi = \psi$ for $p \notin S_f$ imply that $\cT(\psi) \gg_{X,\epsilon} N^{1-\epsilon}$.  We shall estimate $\langle b \psi, \phi \rangle$ by estimating $\langle \cT \psi, b \phi \rangle$.  If $g \in G_\infty$, we define the integral

\bes
I(t, \phi, g) = \iint_\ga b\phi(H_1) \overline{b\phi(H_2)} k_t( \exp(-H_1) g \exp(H_2) ) dH_1 dH_2.
\ees
We suppose that $b$ and $h_t$ are chosen so that $I(t, \phi, g) = 0$ unless $d(g,e) \le 1$.  The main amplification inequality that we shall use is the following.

\begin{lemma}
\label{amplemma}

We have

\be
\label{phiprod}
| \langle \cT \psi, b \phi \rangle |^2 \le \sum_{\gamma \in \bG(\Q)} | (\cT \cT^*)(\gamma) I(t, \phi, g_0^{-1} \gamma g_0) |.
\ee

\end{lemma}

\begin{proof}

Consider the function

\bes
K(x,y) = \sum_{\gamma \in \bG(\Q)} \overline{k_t} \cT \cT^*(x^{-1} \gamma y)
\ees
on $\bG(\A) \times \bG(\A)$.  We have

\bes
\iint_\ga \overline{b\phi(H_1)} b\phi(H_2) K( g_0 \exp(H_1), g_0 \exp(H_2) ) dH_1 dH_2 = \sum_{\gamma \in \bG(\Q)} (\cT \cT^*)(\gamma) \overline{I(t, \phi, g_0^{-1} \gamma g_0)}.
\ees
We write the spectral decomposition of $L^2(X)$ as

\bes
L^2(X) = \int V_\pi d\pi,
\ees
where each $V_\pi$ is a one dimensional space spanned by an eigenform $\psi_\pi$ (which may be an Eisenstein series or an iterated residue thereof).  We assume that $\psi \in \{ \psi_\pi\}$.  If we let the spectral paramater of $\psi_\pi$ be $\lambda_\pi$, the identity $k_t = k_t^*$ implies that

\begin{align*}
\pi(\overline{k_t}) \psi_\pi & = \psi_\pi \int_{G_\infty} \overline{k_t}(g) \varphi_{\lambda_\pi}(g) d\overline{g}_\infty \\
& = \psi_\pi \int_{G_\infty} k_t(g) \varphi_{\lambda_\pi}(g^{-1}) d\overline{g}_\infty \\
& = \psi_\pi \int_{G_\infty} k_t(g) \varphi_{-\lambda_\pi}(g) d\overline{g}_\infty \\
& = h_t(\lambda_\pi) \psi_\pi.
\end{align*}
The spectral expansion of $K(x,y)$ is therefore

\bes
K(x,y) = \int h_t(\lambda_\pi)(\cT \psi_\pi )(x) \overline{(\cT \psi_\pi)(y)} d\pi,
\ees
which implies that

\be
\label{specint}
\iint_\ga \overline{b\phi(H_1)} b\phi(H_2) K( g_0 \exp(H_1), g_0 \exp(H_2) ) dH_1 dH_2 = \int h_t(\lambda_\pi) | \langle \cT \psi_\pi, b \phi \rangle |^2 d\pi.
\ee
Unitarity implies that all spectral parameters in the integral satisfy $h_t(\lambda_\pi) = |h_t^0(\lambda_\pi)|^2 \ge 0$.   Dropping all terms in (\ref{specint}) but $\psi$ and using the inequality $h_t(t\lambda) \ge 1$ gives the lemma.

\end{proof}

We shall estimate (\ref{phiprod}) with the aid of two propositions.  The first says that the only significant contribution to (\ref{phiprod}) comes from those $\gamma \in \bG(\Q) \cap \text{supp}(\cT \cT^*)$ for which $g_0^{-1} \gamma g_0$ is near $MA$.

\begin{prop}
\label{Ibound}

We have

\be
\label{Ibound1}
| I(t, \phi, g) | \ll t^{3/2 + \epsilon} \beta^{9/2}
\ee
for all $g$ and any $\epsilon > 0$.  If $d(g, MA) \ge t^{-1/2 + \epsilon} \beta^{1/2}$ for some $\epsilon > 0$, we have

\be
\label{Ibound2}
| I(t, \phi, g) | \ll t^{-A}.
\ee
The implied constants depend on $B^*$, $\epsilon$, and $A$ where applicable.

\end{prop}

The second proposition bounds the number of isometries that map $E$ close to itself.  If $x \in R$ is nonzero, denote its image in $\bG(\Q)$ by $\overline{x}$.  If $g \in G_\infty$, $a, b, c \in \Q^\times$, $n \in \Z_{>0}$, and $\kappa > 0$, we define

\begin{align*}
\cM(g, a, b, c, \kappa) & = \{ \gamma \in \bG(\Q) \cap K(a,b,c) | d(g^{-1} \gamma g,e) \le 1, d(g^{-1} \gamma g, MA) \le \kappa \}, \\
M(g, a, b, c, \kappa) & = |\cM(g, a, b, c, \kappa)|,
\end{align*}
and

\begin{align*}
\cL(g, n, \kappa) & = \{ x \in R | |\text{nr}(x)| = n, d(g^{-1} \overline{x} g,e) \le 1, d(g^{-1} \overline{x} g, MA) \le \kappa \}, \\
L(g, n, \kappa) & = |\cL(g, n, \kappa)|.
\end{align*}
The following three results allow us to bound $M(g, a, b, c, \kappa)$.

\begin{lemma}
\label{returnlem}

If $a, b, c \in \Z_{>0}$ are relatively prime to $S_f$ and satisfy $(a,b,c) = 1$, we have

\bes
M(g, a, b, c, \kappa) \le L(g, abc, \kappa).
\ees

\end{lemma}

\begin{proof}

Suppose $\gamma \in \cM(g, a, b, c, \kappa)$.  If $x \in R$ is a primitive element that projects to $\gamma$, then $x_p$ must be primitive in $R_p$ for all $p$.  We have

\begin{align*}
x_p & \in K_p Z_p \subseteq R_p^\times Z_p, \quad p \in S_f, \\
x_p & \in K_p( \ord_p(a), \ord_p(b), \ord_p(c)) Z_p, \quad p \notin S_f.
\end{align*}
The primitivity of $x_p$ and the condition $(a,b,c) = 1$ imply that

\begin{align*}
x_p & \in R_p, \quad p \in S_f, \\
x_p & \in K_p( \ord_p(a), \ord_p(b), \ord_p(c)), \quad p \notin S_f,
\end{align*}
so that $|\text{nr}(x)| = abc$.  It follows that $x \in \cL(g,abc,\kappa)$.  Because distinct elements of $\cM(g, a, b, c, \kappa)$ have distinct elements of $\cL(g,abc,\kappa)$ assigned to them, the result follows.

\end{proof}

\begin{prop}
\label{returnprop}

There is $C = C(\Omega) > 0$ such that if $g \in \Omega$ and $\kappa \le C n^{-24}$, then $L(g, n, \kappa ) \ll_{\Omega, \epsilon} n^\epsilon$.

\end{prop}

\begin{proof}

Suppose that $x \in \cL(g, n, \kappa)$.  The assumptions that $d(g^{-1} \overline{x} g, e) \le 1$ and $g \in \Omega$ imply that $\overline{x}$ lies in a compact set $\Omega' \subset G_\infty$ depending on $\Omega$.  Combined with the assumption that $|\text{nr}(x)| = n$, this gives $\| x \| \ll_\Omega n^{1/3}$.

The assumption $n(g^{-1} \overline{x} g, MA) \le \kappa$ gives

\be
\label{gammaclose}
\text{inf} \{ \| x - z \| \, | \, z \in g Z_A g^{-1} \} \le C' \kappa n^{1/3}
\ee
for some $C'$ depending only on $\Omega$.  Lemma \ref{Akshay} below implies that there is $C > 0$ depending on $\Omega$ such that if $\kappa \le C n^{-24}$, then $\cL(g,n,\kappa)$ must be contained in a cubic subfield $F \subset D$.

As $x \in R \cap F$, $x$ must lie in the ring of integers $\cO_F$ of $F$, and as the reduced norm $\text{nr}$ on $D$ agrees with the norm $N$ on $F$ we must have $|N(x)| = n$.  The condition $\| x \| \ll_\Omega n^{1/3}$ implies that the image of $x$ under all archimedean embeddings of $F$ must be $\ll_\Omega n^{1/3}$, and the number of algebraic integers in $F$ satisfying these conditions may easily be seen to be $\ll_{\Omega,\epsilon} n^\epsilon$, uniformly in $F$.

\end{proof}

The following lemma is taken from an unpublished paper of Lior Silberman and Akshay Venkatesh.  We thank the authors for permission to reproduce it here.

\begin{lemma}
\label{Akshay}

Let $S \subset D_\infty$ be an $\R$-subalgebra isomorphic to $\R^3$.  For $C > 0$ sufficiently small depending on $R$, and any $X > 1$ and $0 < \epsilon < 1$, the set of $x \in R$ satisfying

\be
\label{xS}
\| x \| \le X, \quad \inf \{ \| x - s \| \, | \, s \in S \} \le \epsilon
\ee
is contained in a subalgebra $F \subset D$ of dimension at most 3 as long as

\be
\label{CC}
\epsilon X^{71} < C.
\ee

\end{lemma}

\begin{proof}

Fix an integral basis $\{ v_i \}_{i=1}^9$ of $R$.  If $(\alpha_1, \alpha_2, \alpha_3, \alpha_4) \in D^4$, define $P(\alpha_1, \ldots, \alpha_4)$ by expressing each $\alpha_i$ in the basis $\{ v_i \}$, combining the resulting vectors into a $4 \times 9$ matrix, and taking the sums of the squares of the $4 \times 4$ minors.  This makes $P : D^4 \rightarrow \Q$ a polynomial map of degree 8, with integral coefficients with respect to $R$, such that $P(\alpha_1, \ldots, \alpha_4) = 0$ exactly when $\alpha_1, \ldots, \alpha_4$ span a linear subspace of dimension $\le 3$.

Let $F$ be the $\Q$-algebra generated by those $x \in R$ satisfying (\ref{xS}).  It is clear that $F$ is in fact spanned by monomials in such $x$ of length at most $9 = \dim_\Q D$.  Each such monomial $y$ lies in $R$, and satisfies the bounds

\be
\label{xS2}
\| y \| \ll X^9, \quad \inf \{ \| y - s \| \, | \, s \in S \} \ll X^8 \epsilon.
\ee
Take $y_1, \ldots, y_4 \in R$ satisfying (\ref{xS2}).  There are $z_1, \ldots, z_4 \in S$ such that $\| y_i - z_i \| \ll X^8 \epsilon$, and so $P(y_1, \ldots, y_4) \ll (X^9)^7 X^8 \epsilon = X^{71} \epsilon$.  On the other hand, if $P(y_1, \ldots, y_4) \neq 0$ then we must have $|P(y_1, \ldots, y_4)| \ge 1$ by the condition that $P$ had integral coefficients with respect to $R$.  It follows that if a condition of the form (\ref{CC}) holds for $C$ as stated, then $y_1, \ldots, y_4$ span a $\Q$-linear space of dimension at most 3.  It follows that $F$ has dimension at most 3, which completes the proof.

\end{proof}

With these results, we may now estimate (\ref{phiprod}).  In the calculation occupying the rest of this section, all implied constants will depend on $B^*$, $\Omega$, and $\epsilon$.  We have

\bes
\cT \cT^* = \sum_{ p \neq q \le N } T_p T_q^* + \sum_{p \le N} T_p T_p^*.
\ees
We may bound

\bes
\sum_{\gamma \in \bG(\Q)} \sum_{p \neq q \le N} | T_p T_q^*(\gamma) I(t, \phi, g_0^{-1} \gamma g_0)|
\ees
by the sum of four expressions, the first of which is

\bes
\sum_{p \neq q \le N} \sum_{\gamma \in \bG(\Q)} \frac{1}{pq} \Phi(pq,p,1)(\gamma) |I(t, \phi, g_0^{-1} \gamma g_0)|.
\ees
The bound (\ref{Ibound2}) and our assumption that $N \le t$ allow us to restrict the sum over $\gamma$ to $\cM(g_0, pq, p, 1, t^{-1/2+\epsilon} \beta^{1/2})$.  If we assume that $N \le t^{1/600}$, we may combine Lemma \ref{returnlem}, Proposition \ref{returnprop}, and the bound (\ref{Ibound1}) to obtain

\bes
\sum_{\gamma \in \bG(\Q)} \Phi(pq,p,1)(\gamma) |I(t, \phi, g_0^{-1} \gamma g_0)| \ll t^{3/2+\epsilon} \beta^{9/2}.
\ees
We therefore have

\be
\label{pqbound}
\sum_{p \neq q \le N} \sum_{\gamma \in \bG(\Q)} \frac{1}{pq} \Phi(pq,p,1)(\gamma) |I(t, \phi, g_0^{-1} \gamma g_0)| \ll t^{3/2+\epsilon} \beta^{9/2} \sum_{p \neq q \le N} \frac{1}{pq} \ll t^{3/2+\epsilon} \beta^{9/2}.
\ee

The other three expressions we must consider are similar, but with $\Phi(pq,p,1) / pq$ replaced with $\Phi(p^2 q,p,1) / p^2q$, $\Phi(pq^2,pq,1) / pq^2$, and $\Phi(p^2q^2,pq,1) / p^2q^2$ respectively.  We may bound them in the same way using our assumption that $N \le t^{1/600}$ .

The analysis of the $T_p T_p^*$ terms is similar.  Lemma \ref{Tpexpand} implies that we must consider terms of the form $\Phi(p^a, p^b, 1) / p^a$ with $a \ge b \ge 0$, and the inequalities (\ref{abcbd}) imply that $a + b + c \le 6$.  Arguing as above using the bound $N \le t^{1/600}$ gives

\bes
\sum_{\gamma \in \bG(\Q)} | T_p T_p^*(\gamma) I(t, \phi, g_0^{-1} \gamma g_0)| \ll t^{3/2+\epsilon} \beta^{9/2}.
\ees
If we sum over $p$ and combine this with (\ref{pqbound}), we obtain

\bes
\sum_{\gamma \in \bG(\Q)} | (\cT \cT^*)(\gamma) I(t, \phi, g_0^{-1} \gamma g_0) | \ll N t^{3/2+\epsilon} \beta^{9/2},
\ees
and Lemma \ref{amplemma} gives

\bes
\langle \cT \psi, b \phi \rangle \ll N^{1/2} t^{3/4+\epsilon} \beta^{9/4}.
\ees
The bound $\cT(\psi) \gg_{X,\epsilon} N^{1-\epsilon}$ gives $\langle \psi, b \phi \rangle \ll N^{-1/2} t^{3/4+\epsilon} \beta^{9/4}$.  Choosing $N = t^{1/600}$ completes the proof of Proposition \ref{L2onspec}.

\section{Bounds away from the spectrum}
\label{sec5}

Let $H_\beta$ be as in Section \ref{outline}, and let $\phi \in H_\beta^\perp$ satisfy $\| \phi \|_2 = 1$.  In this section, we prove the following bound for $\langle b \psi, \phi \rangle$.

\begin{prop}
\label{L2offspec}

For any $\epsilon > 0$ there is $C = C(B^*, \Omega, \epsilon) > 0$ such that $\langle b \psi, \phi \rangle \le C t^{3/4} \beta^{-1/4 + \epsilon}$.

\end{prop}

\begin{proof}

If we apply Lemma \ref{amplemma} with $\cT$ chosen to be the characteristic function of a sufficiently small open subgroup of $\bG(\A_f)$, depending on $\Omega$, then only $\gamma = 1$ will make a nonzero contribution to the sum.  This gives

\bes
|\langle \psi, b \phi \rangle|^2 \ll_{\Omega} \left| \iint_\ga b\phi(H_1) \overline{b\phi}(H_2) k_t(\exp(H_2 - H_1)) dH_1 dH_2 \right|.
\ees
If we define $p_t(H) = k_t( \exp(H))$, we may rewrite this as $|\langle \psi, b \phi \rangle|^2 \ll_{\Omega} |\langle p_t *b \phi, b \phi \rangle|$.

For $\nu \in \ga^*$ and $C \ge 0$, we define $B(W \nu, C)$ to be the union of the balls of radius $C$ around the points in $W \nu$ with respect to $\| \cdot \|$.  Write $b \phi = \phi_1 + \phi_2$, where $\widehat{\phi_1}$ is supported on $B(W t\lambda, \beta/2)$ and $\widehat{\phi_2}$ is supported on $\ga^* \setminus B(W t\lambda, \beta/2)$.  Because $b$ was a fixed smooth function, we have $\| \phi_1 \|_2 \ll_{A} \beta^{-A}$.  We have

\begin{align*}
\langle p_t * b \phi, b \phi \rangle & = \langle p_t * \phi_1, \phi_1 \rangle + \langle p_t * \phi_2, \phi_2 \rangle \\
& \le O_{A}(\beta^{-A}) \| \widehat{p_t} \|_\infty + \underset{\mu \notin B(W t\lambda, \beta/2)}{\sup} |\widehat{p_t}(\mu)|.
\end{align*}
The result now follows from the lemma below.

\end{proof}

\begin{lemma}

We have $\| \widehat{p_t} \|_\infty \ll_{B^*} t^{3/2}$, and

\be
\label{phat}
\widehat{p_t}(\mu) \ll_{B^*,\epsilon} t^{3/2} \beta^{-1/2 + \epsilon}
\ee
for $\mu \notin B(W t\lambda, \beta/2)$.

\end{lemma}

\begin{proof}

The first assertion is proven in Section 3 of \cite{Ma2}, or can be deduced easily from the bound

\bes
k_t(\exp(H)) \ll_{B^*} t^3 \prod_{\alpha \in \Delta^+} (1 + t |\alpha(H)| )^{-1/2}
\ees
proven in Lemma 2.6 of \cite{Ma2}.  We shall prove (\ref{phat}) with the aid of two asymptotic formulae for $\varphi_{t\nu}$ taken from \cite{Ma2}.  Let $B_0^*$ be a compact subset of $\ga^*_r$ that contains an open neighbourhood of $B^*$.  In the rest of the proof, all implied constants will depend on $B_0^*$ unless otherwise stated.  For $H \in \ga$ we define 

\bes
\| H \|_s = \inf \{ \| H - Z \| | Z \in \ga_s \}.
\ees
We may apply Theorems 1.3 and 1.4 of \cite{Ma2} to deduce that if $H \in 2 B_\ga$ and $\nu \in B_0^*$, we have

\be
\label{spherical1}
\varphi_{t\nu}(\exp(H)) \ll \prod_{\alpha \in \Delta^+} (1 + t |\alpha(H)| )^{-1/2},
\ee
and there are functions $f_w \in C^\infty((2B_\ga \cap \ga_r) \times B_0^* \times \R_{> 0} )$ for $w \in W$ such that

\be
\label{fwbd}
\left( \frac{\partial}{\partial H} \right)^n f_w(H, \nu, t) \ll_{n} \| H \|_s^{-n} \prod_{\alpha \in \Delta^+} ( t |\alpha(H)| )^{-1/2}
\ee
and

\be
\label{spherical2}
\varphi_{t\nu}(\exp(H)) = \sum_{w \in W} f_w(H, \nu,t) \exp( it\nu(wH)) + O_{A}( (t \| H \|_s)^{-A}) \prod_{\alpha \in \Delta^+} ( t |\alpha(H)| )^{-1/2}.
\ee

Let $\chi$ be the characteristic function of the set of points in $\ga$ that are at distance at most 2 from $\ga_s$.  Fix a function $b_0 \in C^\infty_0(\ga)$ with integral 1 and $\text{supp}(b_0) \subseteq B_\ga$.  If we define $b_1 = \chi * b_0$, we have $b_1(H) = 1$ if $\| H \|_s \le 1$.  Fix a function $b_2 \in C^\infty_0(\ga)$ with $\text{supp}(b_2) \subset 2B_\ga$ that is equal to 1 on $B_\ga$.  Our assumption that $\textrm{supp}(k_t(\exp(H))) \subseteq B_\ga$ implies that

\bes
\widehat{p_t}(\mu) = \int_\ga b_2(H) k_t(\exp(H)) e^{-i\mu(H)} dH.
\ees
We have the trivial bound

\be
\label{phat1}
\int_\ga b_2(H) \varphi_{t\nu}(\exp(H)) e^{-i\mu(H)} dH \ll 1,
\ee
where the implied constant is independent of $B_0^*$, and (\ref{spherical1}) gives

\be
\label{phat2}
\int_\ga b_2(H) \varphi_{t\nu}(\exp(H)) e^{-i\mu(H)} dH \ll t^{-3/2}
\ee
for $\nu \in B_0^*$.  If we could prove that

\be
\label{phat3}
\int_\ga b_2(H) \varphi_{t\nu}(\exp(H)) e^{-it\mu(H)} dH \ll_{\epsilon} t^{-3/2} \beta^{-1/2 + \epsilon}
\ee
for $\mu \notin B(W \lambda, \beta/2t)$ and $\nu \in B(W \lambda, \beta/4t)$, then (\ref{phat}) would follow by combining (\ref{phat1}) -- (\ref{phat3}), inverting the Harish-Chandra transform, and using the rapid decay of $h_t$ away from the set $W t \lambda$ as in the proof of Lemma 4.4. of \cite{Ma1} or Lemma 2.6 of \cite{Ma2}.

We decompose the LHS of (\ref{phat3}) as

\begin{multline*}
\int_\ga b_2(H) b_1( \beta H) \varphi_{t\nu}(\exp(H)) e^{-it\mu(H)} dH \\
+ \sum_{n = 1}^\infty \int_\ga b_2(H)( b_1( 2^{-n} \beta H)  - b_1( 2^{-n+1} \beta H))) \varphi_{t\nu}(\exp(H)) e^{-it\mu(H)} dH.
\end{multline*}
The bound (\ref{spherical1}) implies that the first integral is $\ll_{\epsilon} t^{-3/2}\beta^{-1/2 + \epsilon}$ as in Section 2.5 of \cite{Ma2}.  As $\nu \in B_0^*$ for $t$ large, we may estimate the second integral by applying (\ref{spherical2}).   We shall only consider the term $w = e$, as the others are identical.  We wish to estimate

\bes
\int_\ga b_2(H)(b_1( 2^{-n} \beta H)  - b_1( 2^{-n+1} \beta H))) f_e(H,\nu,t) \exp( it(\nu - \mu)(H) ) dH.
\ees
After replacing $H$ with $2^{n} \beta^{-1} H$ this becomes

\be
\label{dyadicint}
2^{2n} \beta^{-2} \int_\ga b_2( 2^n \beta^{-1} H)(b_1(H)  - b_1( 2H))) f_e(2^n \beta^{-1} H,\nu,t) \exp( it 2^n \beta^{-1} (\nu - \mu)(H) ) dH.
\ee
Our assumptions on $\mu$ and $\nu$ imply that $t 2^n \beta^{-1} \| \nu - \mu \| \ge 2^{n-2}$, and (\ref{fwbd}) implies that all derivatives of the amplitude factor in (\ref{dyadicint}) are

\bes
\ll t^{-3/2} \beta^{3/2} 2^{-3n/2} \prod_{\alpha \in \Delta^+} |\alpha(H)|^{-1/2}.
\ees
Integrating by parts once, we obtain

\begin{align*}
(\ref{dyadicint}) & \ll t^{-3/2} \beta^{-1/2} 2^{-n/2} \int b_2( 2^n \beta^{-1} H)(b_1(H)  - b_1( 2H))) \prod_{\alpha \in \Delta^+} |\alpha(H)|^{-1/2} dH \\
& \ll_{\epsilon} t^{-3/2} \beta^{-1/2 + \epsilon} 2^{(-1/2 -\epsilon)n}.
\end{align*}
Summing over $n$ completes the proof.

\end{proof}

\section{Oscillatory integrals}
\label{sec6}

We now prove the bounds of Proposition \ref{Ibound} for $I(t, \phi, g)$.  We continue to use the notation introduced earlier in the paper, with the exceptions that $B^*$ now denotes a fixed compact subset of $\ga^*_r$ that contains an open neighbourhood of the set $B^*$ used earlier, and we now set $K = SO(3)$.  If $\nu \in \ga^*$, we let $H_\nu \in \ga$ denote the element dual to it under the Killing form.  We write the Iwasawa decomposition on $SL_3(\R)$ as

\be
\label{Iwasawa}
g = n(g) \exp(A(g)) k(g) = \exp(N(g)) \exp(A(g)) k(g).
\ee
We begin by calculating a certain derivative of the function $A$.

\begin{lemma}
\label{Adiff}

Let $g \in SL_3(\R)$ have Iwasawa decomposition $g = nak$.  If $H_1$, $H_2 \in \ga$, then

\bes
\frac{\partial}{\partial s} \langle H_1, A( g \exp(sH_2)) \rangle \Big|_{s=0} = \langle H_1, \textup{Ad}(k) H_2 \rangle.
\ees

\end{lemma}

\begin{proof}

We have

\begin{align*}
\frac{\partial}{\partial s} \langle H_1, A( g \exp(sH_2)) \rangle \Big|_{s=0} & = \frac{\partial}{\partial s} \langle H_1, A( k \exp(sH_2)) \rangle \Big|_{s=0} \\
& = \frac{\partial}{\partial s} \langle H_1, A( \exp(s \Ad(k) H_2)) \rangle \Big|_{s=0} \\
& = \langle H_1, \Ad(k) H_2 \rangle.
\end{align*}

\end{proof}

For $g \in SL_3(\R)$, let $\Phi_g : K \rightarrow K$ be the map sending $k$ to $k(kg)$.  

\begin{lemma}

$\Phi_g$ is a diffeomorphism that depends smoothly on $g$.

\end{lemma}

\begin{proof}

The smoothness of the Iwasawa decomposition implies that $\Phi_g$ is smooth, and depends smoothly on $g$.  The inclusion

\be
\label{phicond}
\Phi_g(k) g^{-1} \in NAk
\ee
implies that $\Phi_g$ is injective.  To prove surjectivity, let $k_0 \in K$ and define $k_1 \in K$ by $k_0 g^{-1} \in NA k_1$.  The condition (\ref{phicond}) implies that $\Phi_g(k_1) g^{-1} \in NA k_1 = NA k_0 g^{-1}$.  This implies that $\Phi_g(k_1) \in NA k_0$, so $\Phi_g(k_1) = k_0$ and $\Phi_g$ is surjective.  If $\Phi_g^{-1}$ is the inverse function, (\ref{phicond}) gives $k g^{-1} \in NA \Phi_g^{-1}(k)$ which implies that $\Phi_g^{-1}$ is smooth.

\end{proof}

If $g \in NA$ and $m \in M$, we have

\be
\label{mphi}
\Phi_g(m) = m, \quad \Phi_g(mk) = m \Phi_g(k).
\ee
We now estimate two integrals that comprise $I(t, \phi, g)$.

\begin{prop}
\label{aintprop1}

Fix $C$, $\epsilon > 0$, a compact set $D_B \subset NA$, and a function $b \in C^\infty_0(\ga)$ with $\textup{supp}(b) \subseteq B_\ga$.  If $k \in K$ and $\lambda$, $\nu \in B^*$ satisfy $\| \lambda - \nu \| \le \beta/t$ and

\be
\label{aintass1}
k \notin M \exp( C t^{-1/2 + \epsilon} \beta^{1/2} B_\gk),
\ee
and $g \in D_B$, then we have

\be
\label{aint1}
\int_\ga b(H) \exp( -i t\lambda(H) + i t\nu( A( k g \exp(H) )) ) dH \ll_{A} t^{-A}.
\ee
The implied constant depends on $B^*$, $D_B$, $C$, $\epsilon$, and the size of the derivatives of $b$ of order at most $n$, where $n$ depends on $\epsilon$ and $A$.

\end{prop}

\begin{proof}

Define

\bes
\psi(H, k, g) = \lambda(H) - \nu( A( k g \exp(H) )), \quad H \in B_\ga, k \in K, g \in D_B,
\ees
to be the phase of the integral (\ref{aint1}).  Throughout the proof, the variables $g$ and $H$ will always be restricted to $D_B$ and $B_\ga$ respectively.  We shall always let $H_\nu \psi$ denote the derivative of $\psi$ with respect to the variable $H$.  Lemma \ref{Adiff} implies that

\begin{align}
\notag
H_\nu \psi(H, k, g) & = \langle H_\nu, H_\lambda \rangle - \langle H_\nu, \Ad( \Phi_{g \exp(H)}(k) ) H_\nu \rangle \\
\label{psidiff}
& = \langle H_\nu, H_\lambda - H_\nu \rangle + \langle H_\nu, (1 - \Ad( \Phi_{g \exp(H)}(k) )) H_\nu \rangle.
\end{align}

It is proven in Proposition 5.4 of \cite{DKV} that the function $\langle H_\nu, \Ad(k) H_\nu \rangle$ has a critical point at the identity, and it follows from Proposition 6.5 of \cite{DKV} or Proposition 4.4 of \cite{Ma2} that the Hessian at this critical point is negative definite, uniformly for $\nu \in B^*$.  It follows that, given $B^*$, we may choose an open neighbourhood $0 \in U^1_\gk \subset \gk$ such that for $X \in U^1_\gk$ we have

\be
\label{AHess}
\langle H_\nu, H_\nu \rangle - \langle H_\nu, e^{\ad(X)} H_\nu \rangle \sim_{B^*} \| X \|^2,
\ee
and $\exp$ gives a diffeomorphism $U^1_\gk \rightarrow U_1 := \exp(U^1_\gk)$.  Let $0 \in U_\gk \subset \gk$ be an open set such that $\exp$ gives a diffeomorphism $U_\gk \rightarrow U := \exp(U_\gk)$, and

\be
\label{phicap}
U \subseteq \bigcap_{h \in D_B \exp(B_\ga)} \Phi_h^{-1}(U_1).
\ee

Assume that $k \notin M U$.  Theorem 8.2 of \cite{Ko} implies that if $k_0 \in K$, we have $\langle H_\nu, \Ad(k_0) H_\nu \rangle \le \langle H_\nu, H_\nu \rangle$ with equality iff $k_0 \in M$.  It follows from this and (\ref{mphi}) that there exists $\delta > 0$ depending only on $B^*$ and $D_B$ such that

\be
\label{psinonstat}
\langle H_\nu, (1 - \Ad( \Phi_{g \exp(H)}(k) )) H_\nu \rangle > \delta.
\ee
The condition $\| \lambda - \nu \| \le \beta/t$ implies that

\be
\label{Hnu}
| \langle H_\nu, H_\lambda - H_\nu \rangle | \ll_{B^*} \beta/t \le t^{-1/2},
\ee
and combined with (\ref{psidiff}) and (\ref{psinonstat}) this implies that $H_\nu \psi \ge \delta/2$ for $t$ sufficiently large.  The bound (\ref{aint1}) now follows by integration by parts.

Now assume that $k \in MU$.  In the rest of the proof, all implied constants will depend on $B^*$ and $D_B$.  As $\psi(H, mk, g) = \psi(H, k, g)$ for $m \in M$, we may assume that $k = \exp(X) \in U$ with $X \in U_\gk$.  The assumption (\ref{aintass1}) implies that

\be
\label{Xnorm}
\| X \|^2 \ge C^2 t^{-1+2\epsilon} \beta.
\ee
The condition (\ref{phicap}) on $U$ implies that $\Phi_{g \exp(H)}(k) \in U_1$, and so we may define $Y(H, k, g) \in U^1_\gk$ by requiring that $\Phi_{g \exp(H)}(k) = \exp(Y(H, k, g))$.  Because $\Phi_{g \exp(H)}$ is a smoothly varying family of diffeomorphisms fixing $e$, we have $\| Y(H, k, g) \| \sim \| X \|$, and hence (\ref{AHess}) gives

\begin{align*}
\langle H_\nu, (1 - \Ad( \Phi_{g \exp(H)}(k) )) H_\nu \rangle & = \langle H_\nu, (1 - e^{\textrm{ad}(Y(H,g,k))} ) H_\nu \rangle \\
& \gg \| X \|^2.
\end{align*}
It follows by combining this with (\ref{Hnu}) and (\ref{Xnorm}) that $H_\nu \psi \gg \| X \|^2$.  Combining (\ref{psidiff}) and (\ref{AHess}) gives

\bes
H_\nu \psi(H, k, g) - \langle H_\nu, H_\lambda - H_\nu \rangle \ll \| X \|^2,
\ees
which implies that $H_\nu \psi(H, k, g) - \langle H_\nu, H_\lambda - H_\nu \rangle$ vanishes to second order at $k = e$.  It follows that

\bes
H_\nu^n \psi(H, k, g) \ll_n \| X \|^2
\ees
for $n \ge 2$.  As $t \| X \|^2 \ge C^2 t^{2\epsilon} \beta$, the result now follows by integration by parts with respect to $H_\nu$.

\end{proof}

\begin{prop}
\label{aintprop2}

Fix $C > 0$, $1/10 > \epsilon > 0$, a compact set $D_B \subset NA$, and a function $b \in C^\infty_0(\ga)$ with $\textup{supp}(b) \subseteq B_\ga$.  If $g = na \in D_B$ and $\lambda$, $\nu \in B^*$ satisfy $\| \lambda - \nu \| \le \beta/t$ and

\be
\label{nbd}
n \notin \exp(C t^{-1/2+\epsilon} \beta^{1/2} B_\gn),
\ee
then

\be
\label{aint2}
\int_\ga b(H) e^{-it \lambda(H)} \varphi_{t\nu}(g \exp(H)) dH \ll_A t^{-A}.
\ee
The implied constant depends on $B^*$, $D_B$, $C$, $\epsilon$, and the size of the derivatives of $b$ of order at most $n$, where $n$ depends on $\epsilon$ and $A$.

\end{prop}

\begin{proof}

If we substitute the formula for $\varphi_{t\nu}$ as an integral of plane waves into (\ref{aint2}), it becomes

\bes
\int_\ga \int_K b(H) \exp(-it \lambda(H) + (\rho + it\nu)(A(kg \exp(H)))) dk dH.
\ees
Choose a function $f \in C^\infty_0(\gk)$ with $\text{supp}(f) \subseteq 2B_\gk$ and $f(X) = 1$ for $X \in B_\gk$.  Let $C_1 > 0$ be a constant to be chosen later, define $b_1 \in C^\infty(K)$ to be the pushforward of $f( C_1^{-1} t^{1/2-\epsilon} \beta^{-1/2} X)$ under $\exp$, and define $b_2$ by $b_2(k) = 1 - \sum_{m \in M} b_1(mk)$.  If $t$ is sufficiently large depending on $C_1$, we have

\begin{align*}
\text{supp}(b_1) & \subseteq \exp(2 C_1 t^{-1/2+\epsilon} \beta^{1/2} B_\gk), \\
\text{supp}(b_2) & \subseteq K \setminus M \exp( C_1 t^{-1/2+\epsilon} \beta^{1/2} B_\gk).
\end{align*}
Proposition \ref{aintprop1} and the assumption $\| \lambda - \nu \| \le \beta/t$ imply that

\bes
\int_\ga \int_K b_2(k) b(H) \exp(-it \lambda(H) + (\rho + it\nu)(A(kg \exp(H)))) dk dH \ll_A t^{-A}.
\ees
It therefore suffices to prove that

\bes
\int_\ga \int_K b_1(mk) b(H) \exp(-it \lambda(H) + (\rho + it\nu)(A(kg \exp(H)))) dk dH \ll_A t^{-A}
\ees
for $m \in M$, and we assume without loss of generality that $m = e$.  We shall do this by proving the bound

\be
\label{kintdecay}
\int_K b_1(k) \exp( it\nu(A(kg))) dk \ll_A t^{-A}
\ee
for the integrals over $K$.  Note that we have absorbed $\exp(H)$ into $g$ after enlarging $D_B$, and absorbed the factor $\exp( \rho( A(kg \exp(H))))$ into $b_1(k)$ and suppressed the dependence on $H$ and $g$.  The implied constants in the above three bounds have the same dependence as (\ref{aint2}).

We define the phase functions

\bes
\psi(k,g) = \nu(A(kg)), \quad k \in K, \; g \in NA,
\ees
and

\bes
\quad \psi_S(x) = \nu(A(x)), \quad x \in S.
\ees
If $X \in \g$, we let $X^S$ be the vector field on $S$ whose value at $x \in S$ is $\frac{\partial}{\partial t} \exp(t X) x |_{t=0}$.  It may be shown that these vector fields satisfy $[X^S, Y^S] = -[X,Y]^S$, where the first Lie bracket is on $S$ and the second is in $\g$.  We choose $X_\alpha \in \g_\alpha$ for each $\alpha \in \Delta$ so that the relations $X_{-\alpha} = \theta X_\alpha$ and $\langle X_\alpha, X_{-\alpha} \rangle = -1/2$ are satisfied, where $\theta$ is the Cartan involution on $\g$.  If $\alpha \in \Delta^+$, we let $K_\alpha = X_\alpha + X_{-\alpha} \in \gk$.  It may be seen that the fields $\{ X_\alpha^S | \alpha \in \Delta^+ \}$ form a basis for the normal bundle to $A$ in $S$.

Proposition 5.4 of \cite{DKV} implies that the set of $x \in S$ where the functions $\{ K_\alpha^S \psi_S | \alpha \in \Delta^+ \}$ vanish simultaneously is exactly $A$.  The following lemma shows that these functions in fact form a co-ordinate system transversely to $A$.

\begin{lemma}
\label{XKpsi}

If $\alpha$, $\beta \in \Delta^+$ and $x \in A$, we have $X_\alpha^S K_\beta^S \psi_S(x) = \delta_{\alpha \beta} \langle \alpha, \nu \rangle \psi_S(x) /2$.

\end{lemma}

\begin{proof}

We have

\bes
X_\alpha^S K_\beta^S \psi_S = K_\beta^S X_\alpha^S \psi_S + [X_\alpha^S, K_\beta^S] \psi_S.
\ees
The first term vanishes, as $X_\alpha^S \psi_S \equiv 0$.  If $\alpha = \beta$ then

\bes
[X_\alpha^S, K_\alpha^S] = - [X_\alpha, K_\alpha]^S = \tfrac{1}{2} H_\alpha^S,
\ees
and the lemma follows.  If $\alpha \neq \beta$ then $[X_\alpha, K_\beta ] \in \ga^\perp$, so along $A$ we have

\bes
[X_\alpha, K_\beta]^S \in \text{span} \langle X_\gamma^S | \gamma \in \Delta^+ \rangle.
\ees
As the fields $X_\gamma^S$ annihilate $\psi_S$, the lemma follows.

\end{proof}

Fix a compact set $D_A \subset A$ such that $D_B \subset N D_A$, and a second compact set $D_A' \subset A$ that contains an open neighbourhood of $D_A$.  Let $S_\gn$ be the unit sphere in $\gn$ with respect to $\| \cdot \|$.  We have the following corollary of Lemma \ref{XKpsi}.

\begin{cor}

There exist open sets $U_\alpha' \subset S_\gn$ for each $\alpha \in \Delta^+$, and $\sigma$, $\delta > 0$, such that $S_\gn = \bigcup U_\alpha'$, and if $\alpha \in \Delta^+$ and $x \in \exp( [0, \sigma) U_\alpha') D_A' \subset S$ then

\be
\label{Kpsilower}
|K_\alpha^S \psi_S(x)| \ge \delta \| N(x) \|,
\ee
where $N(x)$ is as in (\ref{Iwasawa}).  The data $U'_\alpha$, $\delta$, and $\sigma$ depend on $B^*$ and $D_A'$.

\end{cor}

Suppose that $g \in \exp( \sigma B_\gn /2) D_A$.  For each $\alpha \in \Delta^+$, we choose a second open set $U_\alpha \subset S_\gn$ such that $\overline{U}_\alpha \subset U_\alpha'$ and we still have $S_\gn = \bigcup U_\alpha$.  Let $\alpha \in \Delta^+$ be such that $g \in \exp( [0, \sigma/2)  U_\alpha) D_A$.  The assumption (\ref{nbd}) is equivalent to the bound $\| N(g) \| \ge C t^{-1/2+\epsilon} \beta^{1/2}$, and we may assume that $k \in \text{supp}(b_1) \subseteq \exp(2 C_1 t^{-1/2+\epsilon} \beta^{1/2} B_\gk)$.  These imply that if $C_1$ is sufficiently small depending on $C$ and $D_B$, we have $kg \in \exp( [0, \sigma)  U_\alpha') D_A'$, and there is $C_2 = C_2(C, C_1, D_B) > 0$ such that $\| N(kg) \| \ge C_2 t^{-1/2+\epsilon} \beta^{1/2}$.  It follows from (\ref{Kpsilower}) that

\bes
|K_\alpha \psi(k,g)| = |K_\alpha^S \psi_S(kg) | \ge \delta \| N(kg) \| \ge \delta C_2 t^{-1/2+\epsilon} \beta^{1/2}
\ees
for $k \in \text{supp}(b_1)$.  The bound (\ref{kintdecay}) now follows from $K_\alpha^n \psi(k,g) \ll C_3(B^*, D_B, n)$, and an application of Lemma \ref{intparts} below with $\delta = t^{-1/2 + \epsilon} \beta^{1/2}$.

Suppose that $g \in D_B \setminus \exp( \sigma B_\gn /2) D_A$.  The fact that the simultaneous vanishing set of the functions $K_\alpha^S \psi_S$ is exactly $A$ implies that there exists an $\alpha \in \Delta^+$, and $C_4(B^*, D_B, \sigma) > 0$, such that $| K_\alpha \psi(k,g) | \ge C_4$ for $k \in \text{supp}(b_1)$ and $C_1$ sufficiently small depending on $B^*$, $D_B$, and $\sigma$.  The proposition again follows by integration by parts.

\end{proof}

\begin{lemma}
\label{intparts}

If $b \in C^\infty_0(\R)$ is a cutoff function at scale $1 \ge \delta > 0$, and $\phi \in C^\infty(\R)$ satisfies

\bes
\phi'(x) \gg \delta, \quad \phi^{(n)}(x) \ll_n 1
\ees
for $x \in \textup{supp}(b)$ and $n > 1$, then

\bes
\int b(x) e^{it\phi(x)} dx \ll_A \delta (t \delta^2)^{-A}.
\ees

\end{lemma}

\begin{proof}

We have

\bes
\int b(x) e^{it\phi(x)} dx = \delta \int b(\delta x) e^{i(t \delta^2) \delta^{-2} \phi(\delta x)} dx.
\ees
The function $b(\delta x)$ is now a cutoff function at scale 1, and $\delta^{-2} \phi(\delta x)$ satisfies

\bes
\frac{d}{dx} \delta^{-2}\phi(\delta x) \gg 1, \quad \left( \frac{d}{dx} \right)^n \delta^{-2}\phi(\delta x) \ll_n \delta^{n-2} \le 1
\ees
for $x \in \textup{supp}(b(\delta \, \cdot))$ and $n > 1$.  The lemma now follows by integration by parts.

\end{proof}

We now combine Propositions \ref{aintprop1} and \ref{aintprop2} to prove the following, which will imply Proposition \ref{Ibound} after inverting the various transforms.

\begin{prop}
\label{aaintprop}

Fix functions $b_1$, $b_2 \in C^\infty_0(\ga)$ with $\textup{supp}(b_i) \subseteq B_\ga$.  If $g  \in SL_3(\R)$ and $\nu$, $\lambda_1$, $\lambda_2 \in B^*$ satisfy 

\bes
d(g,e) \le 1 \quad and \quad \| \lambda_i - \nu \| \le \beta/t,
\ees
then we have

\be
\label{aaint1}
\iint_\ga b_1(H_1) b_2(H_2) e^{it(\lambda_1(H_1) - \lambda_2(H_2))} \varphi_{t\nu}( \exp(-H_1) g \exp(H_2)) dH_1 dH_2 \ll t^{-3/2 + \epsilon} \beta^{3/2}
\ee
for any $\epsilon > 0$.  If we also have $d(g,MA) \ge t^{-1/2+\epsilon} \beta^{1/2}$ for some $\epsilon > 0$, then

\be
\label{aaint2}
\iint_\ga b_1(H_1) b_2(H_2) e^{it(\lambda_1(H_1) - \lambda_2(H_2))} \varphi_{t\nu}( \exp(-H_1) g \exp(H_2)) dH_1 dH_2 \ll_A t^{-A}.
\ee
The implied constants depend only on $B^*$, $\epsilon$, and the size of the derivatives of $b_i$ of order at most $n$, where $n$ depends on $\epsilon$ and $A$.

\end{prop}

\begin{proof}

If we unfold the integral over $K$ defining $\varphi_{t\nu}$ in the LHS of (\ref{aaint1}), we obtain the integral

\begin{multline}
\label{aakint}
\iint_\ga \int_K b_1(H_1) b_2(H_2) e^{it(\lambda_1(H_1) - \lambda_2(H_2))} \exp( (\rho + it\nu)( A( k \exp(-H_1) g \exp(H_2) ))  dk dH_1 dH_2.
\end{multline}
Choose constants $C_1, \epsilon > 0$, and assume that the variable $k$ in (\ref{aakint}) satisfies $k \notin M \exp( C_1 t^{-1/2+\epsilon} \beta^{1/2} B_\gk)$.  After writing $g \exp(H_2) = n'a'k'$, we see Proposition \ref{aintprop1} implies that the partial integral of (\ref{aakint}) with respect to $H_1$ is $\ll_A t^{-A}$ for this fixed value of $k$ and any $H_2 \in B_\ga$.  This implies that we may restrict the integral over $K$ in (\ref{aakint}) to $M \exp( C_1 t^{-1/2+\epsilon} \beta^{1/2} B_\gk)$, and this gives the bound (\ref{aaint1}).

We now prove (\ref{aaint2}).  Let $\epsilon > 0$ be given, and assume that $d(g, MA) \ge t^{-1/2+\epsilon} \beta^{1/2}$.  We may assume without loss of generality that $\epsilon < 1/10$.  Let $g = k_1 n_1 a_1$, and let

\bes
\exp(H) g = k_1(H) n_1(H) a_1(H)
\ees
for $H \in \ga$.  Choose a constant $C_2 > 0$, and assume that $k_1 \notin M \exp( C_2 t^{-1/2+\epsilon} \beta^{1/2} B_\gk)$.  Then there is an absolute constant $\kappa > 0$ such that we have

\be
\label{k1notin}
k_1(H) \notin M \exp( \kappa C_2 t^{-1/2+\epsilon} \beta^{1/2} B_\gk)
\ee
for $H \in B_\ga$.  Let $k \in K$, and consider the integral

\begin{multline}
\label{aaint}
\iint_\ga b_1(H_1) b_2(H_2) e^{it(\lambda_1(H_1) - \lambda_2(H_2))} \exp( (\rho + it\nu)( A( k \exp(-H_1) g \exp(H_2) )) dH_1 dH_2.
\end{multline}
The condition (\ref{k1notin}) implies that if $C_1$ is sufficiently small depending on $C_2$, one of

\begin{align}
\label{kout1}
k & \notin M \exp( C_1 t^{-1/2+\epsilon} \beta^{1/2} B_\gk) \\
\label{kout2}
k k_1(-H_1) & \notin M \exp( C_1 t^{-1/2+\epsilon} \beta^{1/2} B_\gk) \quad \textrm{for all } H_1 \in B_\ga
\end{align}
must hold.  If (\ref{kout1}) holds, Proposition \ref{aintprop1} applied to the integral of (\ref{aaint}) over $H_1$ implies that $(\ref{aaint}) \ll_A t^{-A}$.  If (\ref{kout2}) holds, Proposition \ref{aintprop1} applied to the integral over $H_2$ also implies $(\ref{aaint}) \ll_A t^{-A}$.  Integrating these bounds over $K$ gives (\ref{aaint2}).

We may therefore assume that $k_1 \in M \exp( C_2 t^{-1/2+\epsilon} \beta^{1/2} B_\gk)$ for any $C_2 > 0$, and hence that

\be
\label{gangle}
k_1(H) \in M \exp( C_2 t^{-1/2+\epsilon} \beta^{1/2} B_\gk)
\ee
for any $C_2 > 0$ and $H \in B_\ga$.  Our assumption that $d(g, MA) \ge t^{-1/2+\epsilon} \beta^{1/2}$ implies that there is an absolute $C_3 > 0$ such that

\be
\label{gsep}
d( \exp(H)g, MA) \ge C_3 t^{-1/2+\epsilon} \beta^{1/2}
\ee
for $H \in B_\ga$.  If $C_2$ is sufficiently small, (\ref{gangle}) and (\ref{gsep}) imply that there is an absolute $C_4 > 0$ such that $d(n_1(H),e) \ge C_4 t^{-1/2+\epsilon} \beta^{1/2}$ for $H \in B_\ga$.  The result now follows by applying Proposition \ref{aintprop2} to the integral of the LHS of (\ref{aaint2}) over $H_2$ for each fixed $H_1$.

\end{proof}

\begin{cor}

Fix functions $b_1$, $b_2 \in C^\infty_0(\ga)$ with $\textup{supp}(b_i) \subseteq B_\ga$.  Let $\nu \in B^*$, and choose $\phi \in H(t\nu, \beta)$ with $\| \phi \|_2 = 1$.  If $g \in SL_3(\R)$ satisfies $d(g, e) \le 1$, we have

\bes
\iint_\ga b\phi(H_1) \overline{b\phi(H_2)} \varphi_{t\nu}( \exp(-H_1) g \exp(H_2)) dH_1 dH_2 \ll t^{-3/2 + \epsilon} \beta^{9/2}
\ees
for any $\epsilon > 0$.  If in addition we have $d(g, MA) \ge t^{-1/2+\epsilon} \beta^{1/2}$ for some $\epsilon > 0$, then

\bes
\iint_\ga b\phi(H_1) \overline{b\phi(H_2)} \varphi_{t\nu}( \exp(-H_1) g \exp(H_2)) dH_1 dH_2 \ll_A t^{-A}.
\ees
The implied constants depend only on $B^*$, $\epsilon$, and the size of the derivatives of $b_i$ of order at most $n$, where $n$ depends on $\epsilon$ and $A$.

\end{cor}

\begin{proof}

The follows immediately from Proposition \ref{aaintprop} after inverting the Fourier transform and noting that $\| \widehat{\phi} \|_1 \le \| \widehat{\phi} \|_2 (4\pi \beta^3/3)^{1/2} \ll \beta^{3/2}$.

\end{proof}

Proposition \ref{Ibound} now follows by inversion of the Harish-Chandra transform as in Section 6.3 of \cite{Ma1} or Lemma 2.6 of \cite{Ma2}.

\end{document}